\documentclass{amsart}
\newtheorem{theorem}{Theorem}[section]
\newtheorem{lemma}[theorem]{Lemma}
\newtheorem{corollary}[theorem]{Corollary}
\newtheorem{proposition}[theorem]{Proposition}

\theoremstyle{definition}
\newtheorem{definition}[theorem]{Definition}
\newtheorem{example}[theorem]{Example}

\theoremstyle{remark}
\newtheorem{remark}[theorem]{Remark}
\numberwithin{equation}{section}

\usepackage{color,graphicx,amssymb}
\usepackage{graphicx}
\allowdisplaybreaks

\begin{document}
\title{Pseudo-Orbits, Stationary Measures and Metastability}
\author{Wael Bahsoun}
    \address{Department of Mathematical Sciences, Loughborough University, 
Loughborough, Leicestershire, LE11 3TU, UK}
\email{W.Bahsoun@lboro.ac.uk}
\author{Huyi Hu}
\address{Mathematics Department, Michigan State University, East Lansing, MI 48824, USA}
\email{hu@math.msu.edu}
\author{Sandro Vaienti}
\address{
UMR-6207 Centre de Physique Th\'{e}orique, CNRS, Universit\'{e}
d'Aix-Marseille I, II, Universit\'{e} du Sud, Toulon-Var and FRUMAM,
F\'{e}d\'{e}ration de Recherche des Unit\'{e}s des Math\'{e}matiques de Marseille,
CPT Luminy, Case 907, F-13288 Marseille CEDEX 9}
\email{vaienti@cpt.univ-mrs.fr}
\thanks{H. Hu thanks the University of Toulon and CPT Marseille where parts of this work was conducted. WB would like to thank the London Mathematical Society for supporting his visit to CPT Marseille where parts of this research was conducted. SV thanks the French ANR Perturbations,  the CNRS-PEPS ``Mathematical Methods for Climate Models'' and the CNRS Pics ``Propri\'et\'es statistiques des syst\'emes d\'eterministes et al\'eatoires" N. 05968, for support. SV thanks Ph. Marie for useful discussions, and D. Faranda for his help in producing Figure 3.}

\date{\today}

%
\begin{abstract} We characterize absolutely continuous stationary measures (acsms) of randomly perturbed dynamical systems in terms of pseudo-orbits linking the ergodic components of absolutely continuous invariant measures (acims) of the unperturbed system. We focus on those components, called least-elements , which attract pseudo orbits. Under the assumption that the transfer operators of both systems, the random and the unperturbed, satisfy a uniform Lasota-Yorke inequality on a suitable Banach space, we show that each least element is in a one-to-one correspondence with an ergodic acsm of the random system.
\end{abstract}
\maketitle
\section{Introduction}
 In this paper we study statistical aspects of random perturbations under the assumption that the transfer operators of the random and the unperturbed systems satisfy a uniform Lasota-Yorke inequality on a suitable Banach space (see subsection \ref{Banach}) \footnote{See \cite {HV} and \cite{Li} for recent results and for an exhaustive list of references on deterministic expanding maps satisfying such inequality.}. Let $T:M\to M$, $M\subset \mathbb{R}^q$. A  random orbit $\{x_n^{\varepsilon}\}_n$  is a random process where, for all $n$, $x_{n+1}^{\varepsilon}$ is a random variable whose possible values are obtained in an $\varepsilon$-neighbourhood of $Tx_n^{\varepsilon}$ according to a transition probability ${P}_{\varepsilon}(x_n,.)$. We consider the case where ${P}_{\varepsilon}(x_n,.)$ is absolutely continuous with respect to Lebesgue measure $m$. Denoting the density of the transition probability by $p_{\varepsilon}(x,.)$, we define a perturbed transfer operator , $\mathcal L_{\varepsilon}$, by:
$$
\mathcal{L}_{\varepsilon}f(x)=\int_M p_{\varepsilon}(y,x)f(y)dy
$$
where $f\in L_m^1(M)$. We focus on non-invertible dynamical systems whose transfer operators, perturbed and non-perturbed, satisfy a uniform Lasota-Yorke inequality.  In this paper, the non-perturbed operator
$$\mathcal L f(x)=\sum_{y\in T^{-1}x}\frac{f(y)}{|D_yT)|}$$ is the traditional transfer operator (Perron-Frobenius) associated with the map $T$ \cite{Ba}. Among other things, the Lasota-Yorke inequality implies the existence of a finite number of ergodic acims for the initial system $T$, and a finite number of ergodic acsms for the random system \cite{Ba}.

\bigskip

We then define an equivalence relation between the ergodic components of acims of $T$ using \textit{pseudo-orbits}. Using this equivalence relation, we consequently introduce equivalence classes of ergodic acims. Among the latter, we identify those which attract pseudo-orbits and call them \textit{least elements}. We show that each least element admits a neighbourhood which supports exactly one ergodic acsm of the random system, and the converse is also true, namely the support of any ergodic acsm contains only one least element. This result allows us to identify the {\em attractors} of the random orbits, namely the least elements, and we give a nice illustration in the example 2 in Section 5. Moreover, we use our result to identify random perturbations that exhibit a \textit{metastable behavior.}  Such a phenomenon has recently been a very active topic of research in both ergodic theory \cite{BV1, BV2, DW, FS, GHW, KL2} and applied dynamical systems \cite{DF, SFM}.

\bigskip

Section \ref{main} contains the setup of the problem, our assumptions, the notion of a least element and the statement of our main result (Theorem \ref{MT1}). Section \ref{pf} contains the proof of Theorem \ref{MT1}. In Section \ref{meta} we use the results of the previous sections to identify random systems which exhibit a metastable behavior. In Section \ref{comp} we apply our results to random transformations, in particular, we provide  examples to illustrate the results of Sections \ref{main} and \ref{meta}.
\section{Setup and statement of the main result}\label{main}
\subsection{The initial system}
Let $M\subset \mathbb{R}^q$ be compact\footnote{All our results can be carried to the case where $M$ is a compact Riemannian manifold.} with $\overline{int M}=M$.  We denote by $d$ the Euclidean distance on $ \mathbb{R}^q$. Let $(M, \mathcal{A}, m)$ be the measure space, where $\mathcal A$ is Borel $\sigma$-algebra, and $m$ is the normalized Lebesgue measure on $M$; i.e, $m(M)=1$.\\
 Let $T:M\rightarrow M$ be a measurable map, ${\mathcal D_T}$ be the set of discontinuities of $T$; with the notation ${\mathcal D_g}$ we mean the set of discontinuities of the function $g$. We assume that $m ({\mathcal D_T})=0$, and $T$ is non-singular with respect to $m$. The transfer operator (Perron-Frobenius) \cite{Ba} associated with $T$, $\mathcal{L}:L^1_m\hookrightarrow  L^1_m$ is defined by duality:  for $g_1\in L^1_m$ and $g_2\in L^{\infty}_m$
 $$\int_M g_1 g_2\circ T dm= \int_M \mathcal{L} g_1 g_2 dm.$$
\subsection{The perturbed system}
We perturb the map $T$ by introducing a family of Markov chains $(\mathcal{X}^{\varepsilon}_n)$, $\varepsilon>0, n\ge 0$  with state space $M$ and  transition probabilities $\{P_{\varepsilon}(x,.)\}_{x\in M}$; i.e, $P_{\varepsilon}(x,A)$ is the probability that a point $x$ is mapped into a measurable set $A$. At time $n=0$, $(\mathcal{X}^{\varepsilon}_0)$ can have any probability distribution. We assume that:
\begin{enumerate}
\item[(P1)] For all $x\in M$, $P_{\varepsilon}(x,.)$ is absolutely continuous with respect to the Lebesgue measure. We will denote its density by $p_{\varepsilon}(x,.)$. Therefore
$P_{\varepsilon}(\mathcal{X}^{\varepsilon}_{n+1}\in A|\mathcal{X}^{\varepsilon}_n=x)=\int_A p_{\varepsilon}(x,y)dm(y)$.
\item[(P2)] We have:  supp$P_{\varepsilon}(x,.)= B_{\varepsilon}(Tx)$ for all $x$.
\end{enumerate}
Assumption (P2) can be weakened by supposing that supp$P_{\varepsilon}(x,.)$ is a subset of a slightly larger ball around $Tx$. Our proofs can be easily adapted to show that the results of this paper will still hold under such a slightly weakened assumption.
\subsubsection{Random orbits and stationary measures}
The perturbed evolution of a state $x\in M$ will be represented by a \textit{random orbit}:
\begin{definition}
A sequence $\{x_n^{\varepsilon}\}_{n\ge0}\subset M$ is an \textit{$\varepsilon$-random orbit} if each $x_{n+1}^{\varepsilon}$ is a random variable whose distribution is $P_{\varepsilon}(x_n^{\varepsilon},.)$, namely  $\{x_n^{\varepsilon}\}_{n\ge0}$ coincides with the Markov chains $(\mathcal{X}^{\varepsilon}_n)$, $\varepsilon>0, n\ge 0$.
\end{definition}
The counterpart of an invariant measure in the case of randomly perturbed dynamical systems is called a stationary measure:
\begin{definition}
 A probability measures $\mu_{\varepsilon}$ is called a \textit{stationary measure} if for any $A\in\mathcal A$
$$
\mu_{\varepsilon}(A)=\int_MP_{\varepsilon}(x,A)d\mu_{\varepsilon}(x)
.$$
We call it an absolutely continuous stationary measure ({\em acsm}), if it has a density with respect to the Lebesgue measure (see below).
\end{definition}
\subsubsection{The transfer operator of the random system}
To study Markov processes it is useful to define the transition operator $\mathcal{T}_{\varepsilon}$ acting on bounded real-valued measurable functions $g$ defined on $M$:
$$(\mathcal{T}_{\varepsilon}g)(x)=\int_M g(y)P_{\varepsilon}(x,dy).$$
Its adjoint $\mathcal{T}^{\ast}$ is defined on the space $\mathcal{M}(M)$ of all finite signed measures and  is given by:
$$
\mathcal{T}_{\varepsilon}^{\ast}\mu(A)=\int_MP_{\varepsilon}(x,A)d\mu(x).
$$
The measure $\mu_{\varepsilon}$ is stationary  if and only if $\mathcal{T}_{\varepsilon}^{\ast}\mu_{\varepsilon}=\mu_{\varepsilon}$ \cite{K}. Moreover, a stationary measure is ergodic if, any $A\in \mathcal A$ with $\mathcal{T}_{\varepsilon}1_A=1_A$ implies that $\mu_{\varepsilon}(A)=0$ or $\mu_{\varepsilon}(A)=~1$ \cite{K}.

\bigskip
Since condition (P1) implies the absolute continuity of any stationary measure, it will be convenient to define an operator, $\mathcal{L}_{\varepsilon}$, acting on densities.  That is to say, if $\mu$ is an absolutely continuous measure with respect to $m$, whose density is a function $g\in L^1_m$, then $\mathcal{T}_{\varepsilon}^{\ast}\mu$ is an absolutely continuous measure whose density is $\mathcal{L}_{\varepsilon}g$, where $\mathcal{L}_{\varepsilon}g$ is given by:
\begin{equation}
\mathcal{L}_{\varepsilon}g(x)=\int_Mp_{\varepsilon}(z,x)g(z)dz.
\end{equation}
Thus, densities of acsms are fixed points of $\mathcal{L}_{\varepsilon}$. Results on the existence of acsms can be found in \cite{BaYo} and references therein. We will comment again about this definition of the random transfer operator in Remark \ref{RR}, Section 5.
\subsection{A Banach space and quasi-compactness of $\mathcal L$ and $\mathcal L_{\varepsilon}$}\label{Banach}
We now introduce a Banach space $\mathcal{B}(M)\subset L^1_m$. We assume that
\begin{enumerate}
\item[(B1)]The constant function $1$ belongs to $\mathcal{B}(M)$.
\item[(B2)]The set of discontinuities, $\mathcal D_f$, of any function $f\in \mathcal{B}(M)$ has Lebesgue measure  zero.
\item[(B3)] There is a semi-norm $\left|.\right|$  on $\mathcal{B}(M)$ such that the unit ball of $\mathcal{B}(M)$ is compact in $L^1_m$ with respect to the complete norm $||\cdot ||_B\equiv \left|.\right|+\left\|.\right\|_1$, where $\left\|.\right\|_1$ denotes the $L^1_m$ norm.
\end{enumerate}
We also assume that the transfer operators $\mathcal L$ and $\mathcal L_{\varepsilon}$ satisfy a uniform Lasota-Yorke inequality: there exist an $\eta\in(0,1)$ and a $D\in (0,\infty)$ such that for all $f\in \mathcal{B}(M)$ and $\varepsilon>0$ small enough:
$$|\mathcal{L}f|\leq \eta\left|f\right|+D\left\|f\right\|_1; \hspace{3.5cm} (\normalfont{LY})$$
$$|\mathcal{L}_{\varepsilon}f|\leq \eta\left|f\right|+D\left\|f\right\|_1. \hspace{3.5cm} (\normalfont{RLY})$$
Assumptions (LY) and (RLY) ensure the quasi-compactness of both $\mathcal{L}$ and $\mathcal{L}_{\varepsilon}$, see \cite{Ba} and \cite{AFV} for non-invertible systems.
 In particular, among other things, (LY) implies the existence of a finite number of $T$-ergodic acim, and (RLY) implies the existence of a finite number of ergodic acsm for the Markov process $(\mathcal{X}^{\varepsilon}_n)$. More precisely, we have for the operator $\mathcal {L}$ (see \cite{BG, HH}):
\begin{itemize}
\item The subspace of non-negative fixed points $w$ of $\mathcal{L}$, is a convex set with a finite number of extreme points $w_1,\cdots,w_l$ with supports $\Lambda_k, k=1,\cdots,l$. The supports $\Lambda_k, k=1,\cdots,l$, are mutually disjoint Lebesgue a.e..
\item The measures $\mu_1= w_1 m, \cdots, \mu_l=w_lm$ are ergodic and they give  the ergodic decomposition of any acim $\mu=h m$, $h\in L^1_m$. We also call them the extreme (ergodic) points (measures) decomposing $\mu$.
\end{itemize}
The Lasota-Yorke inequality (RLY) of $\mathcal{L}_{\varepsilon}$ ensures that the random system admits finitely many ergodic acsms $\mu_1^{\varepsilon}=h_{\varepsilon,1}m,\ldots,\mu_K^{\varepsilon}=h_{\varepsilon,K}m$. Note that in general the number $K$ of extreme points for $\mu_{\varepsilon}$ is different from the number $l$ of ergodic components of $\mu$. In our setting, (see Corollary \ref{co1}), we show that the number of ergodic ascms is bounded above by the number of ergodic acims.
\begin{remark}
We point out that for certain perturbations one can prove that (RLY) follows from (LY). See \cite{BaYo, Ke} for precise examples.
\end{remark}
\begin{remark} In our work we are mainly concerned with non-invertible systems. In general, to get spectral results from the Lasota-Yorke inequality, one needs to work with a couple of {\em adapted spaces}. We choose here to work with $ L^1_m$ because it is the partner of the natural space of functions, namely functions of bounded variation, needed when studying the examples of Section 5.  A closer inspection to the proof of Theorem 1 shows that {\em any} pair of adapted spaces will work whenever conditions (B1)-(B3) are satisfied. A generalization of our results to invertible systems will require to replace $L_m^1$ with an appropriate generalized Banach space, see \cite{BaG, DL, DZ} for detailed discussions of such spaces.
\end{remark}
 A well known consequence of assumption (LY) and (RLY) is the following proposition, see for instance \cite{BG, WC}.
 \begin{proposition}\label{prop1}.
Let $\{h_{\varepsilon}\}_{\varepsilon>0}$ be a family of densities of absolutely continuous stationary measures of $(\mathcal X_{\varepsilon}^n)$. Then any limit point, as $\varepsilon\to 0$, of  $\{h_{\varepsilon}\}_{\varepsilon>0}$ in the $L^1_m$-norm is a density of $T$-acim.
\end{proposition}
\subsection{Pseudo-orbits, least elements and the statement of the main result}
We now introduce the notion of a pseudo orbit which will be our main tool to characterize ergodic acsm. Pseudo-orbits were previously used by Ruelle \cite{RU}, followed by Kifer \cite{KI}, to study attractors of randomly perturbed smooth maps. See also \cite{BDV}.
\begin{definition}[pseudo-orbit]
For $\varepsilon>0$, an \textit{$\varepsilon$-pseudo-orbit} is a finite set $\left\{x_i\right\}_{i=0}^n\subset M$ such that $d(Tx_i,x_{i+1})<\varepsilon$ for $i=0\ldots n-1$.
\end{definition}
Using pseudo-orbits, we define a pre-order (reflexive and transitive) ``$\rightarrow$" among the supports $\left\{\Lambda_i\right\}$ of $T$-ergodic acim by writing $\Lambda_i\rightarrow\Lambda_j$ if for any $\varepsilon>0$ there is an $\varepsilon$-pseudo-orbit $\left\{x_i\right\}_{i=0}^k$ such that $x_0\in\Lambda_i$ and $x_k\in \Lambda_j$. Then we define a  relation ``$\sim$" among $\left\{\Lambda_i\right\}$ by writing $\Lambda_i\sim\Lambda_j$ if both $\Lambda_i\rightarrow \Lambda_j$ and $\Lambda_j\rightarrow \Lambda_i$. By ergodicity, given a point  $y\in\Lambda_i\equiv \mbox{supp} \mu_i$ and $\varepsilon>0$,
$\mu_i$-almost any point $x\in \Lambda_i$ will enter the ball $B_{\varepsilon}(y)$ of positive $\mu_i$ measure, and therefore all the pairs $(x,y)\in \Lambda_i$ can be connected with a (finite) $\varepsilon$-pseudo-orbit.  Hence we get an equivalence relation among those ergodic components and we define $\tilde{\Lambda}_i$ the equivalence class which contains $\Lambda_i$. We write
$\tilde{\Lambda}_i\rightarrow\tilde{\Lambda}_j$ if $\Lambda_k\rightarrow \Lambda_l$ for any $\Lambda_k\in\tilde{\Lambda}_i$ and $\Lambda_l\in\tilde{\Lambda}_j$.
\begin{definition}\label{LBS}
 $\tilde{\Lambda}_i$ is said to be a {\em least element},   if there is no $\tilde{\Lambda}_j, \ j\neq i,$ such that $\tilde{\Lambda}_i\rightarrow\tilde{\Lambda}_j$.
\end{definition}
This in particular means that for all $\varepsilon$ small enough no $\varepsilon$-pseudo-orbit can travel from the least element to other equivalence classes. We point out that under the Assumption (LY) and (BLY), there are at most finitely many such equivalence classes $\tilde \Lambda_j$'s, hence, the least elements always exists by Zorn's lemma.  In general, a dynamical system may have more than one least-element. This will be illustrated in Example \ref{Ex1}. We now state our first result:
\begin{theorem} \label{MT1}
Under Assumption (P1) and (P2), if (LY) and (RLY) hold for functions in a Banach space $\mathcal B(M)$ satisfying (B1-B3), then we have:
\begin{enumerate}
    \item  If $\tilde{\Lambda}$ is a least element, then for $\varepsilon$  small enough there exists an open neighborhood ${\mathcal U}_{\varepsilon}\supset \tilde{\Lambda}$ which supports a unique ergodic acsm $\mu_{\varepsilon}^{\tilde{\Lambda}}$.
   \item If $\tilde{\Lambda}$ is not a least element, then for $\varepsilon$ small enough,  $\mu_{\varepsilon}(\tilde{\Lambda})=0$  for any acsm $\mu_{\varepsilon}$. Therefore, for any weak-limit of $\mu_{\varepsilon}$ as $\varepsilon\rightarrow 0$,  $\tilde{\Lambda}$ is a set of measure $0$.
\end{enumerate}
\end{theorem}
Theorem \ref{MT1} implies the following three corollaries:
\begin{corollary}\label{co1}
For $\varepsilon>0$ small enough, the number of acsms of the random system $(\mathcal X_{\varepsilon}^n)$ is bounded above by the number of acims of the map $T$. In particular, if $T$ has a unique acim, then $(\mathcal X_{\varepsilon}^n)$ has a unique acsm.
\end{corollary}
\begin{remark}
In Theorem \ref{MT1} we do not assume that $\underset{{||f||_{\mathcal B}\le 1}}{\sup}||(\mathcal L_{\varepsilon}-\mathcal L)f||_1\to 0$. Thus, Corollary \ref{co1} does not follow directly from the spectral stability result of \cite{KL1}.
\end{remark}
We also have the converse of part 1 of Theorem \ref{MT1}, namely:

\begin{corollary}\label{co3}
For $\varepsilon>0$ small enough, the support of any ergodic acsm contains exactly one least element.
\end{corollary}
Therefore, for $\varepsilon$ small enough,  we can uniquely associate to each least element ${\tilde{\Lambda}}$  the family of densities $\{h_{\varepsilon}^{\tilde{\Lambda}}\}_{\varepsilon>0}$.\\
\begin{definition}
We  say that the system $(M,\mathcal{A}, T)$  is {\em strongly stochastically stable} if any $L^1_m$ limit point  of the densities of the   ergodic absolutely continuous stationary  measures $\{\mu_j^{\varepsilon}\}_{\varepsilon>0}$, $j=1,\ldots, K$ and as $\varepsilon$ goes to zero, is a convex combination of the densities of the absolutely continuous ergodic extreme measures of $\mu$.
\end{definition}In our setting, by Proposition \ref{prop1} and Theorem \ref{MT1}, any limit point of the family $\{h_{\varepsilon}^{\tilde{\Lambda}}\}_{\varepsilon>0}$ as $\varepsilon$ goes to $0$,  is a convex combination of the densities\footnote{Note that within the general setting of this paper, we do not claim that the values of the weights in the convex combination that determine the limiting density can be easily identified. However, for certain perturbations of one dimensional maps, using insights form open dynamical systems, such weights can be determined. See \cite{GHW, BV1, BV2}.} of the  ergodic measures spanning $\tilde{\Lambda}$. Hence we proved that
\begin{corollary}\label{co2}
The system $(M,\mathcal{A}, T)$ is strongly stochastically stable.
\end{corollary}
\begin{remark}
Our definition of stochastic stability is inspired by the  definition in Section 1.1 of \cite{AA}. Whenever there is only one absolutely continuous ergodic invariant measure $\mu$  and only one absolutely continuous ergodic stationary measure $\mu_{\varepsilon}$, it makes sense to speak of strongly stochastic stability of the {\em measure} $\mu$ if the density of $\mu_{\varepsilon}$ converges in $L^1_m$ to the density of $\mu$, see for instance \cite{WC}. Adapting this point of view we can restate the previous corollary by saying that: {\em an absolutely continuous invariant measure for the original system $(M,\mathcal{A}, T)$ whose support is the union of least elements, is strongly stochastically stable}.
\end{remark}
\section{Proofs }\label{pf}
We first prove a key lemma.
\begin{lemma}\label{pos}
Let $f\in \mathcal{B}\equiv \mathcal{B}(M)$ and $\left\{x_i\right\}_{i= 0}^N\subset M$ be an $\varepsilon$-pseudo-orbit such that $x_j\in M\setminus (\mathcal{D}_{\mathcal{L}_{\varepsilon}^jf}\cap \mathcal{D}_T)$, $0\le j\le N$ and $\mathcal{L}_{\varepsilon}^if(x_i)>0$ for some $0\le i<N$.  Then for all $i<k\le N$ we have $\mathcal{L}_{\varepsilon}^kf(x_k)>0$.
\end{lemma}
\begin{proof}
Let $f\in\mathcal B$ and $\left\{x_i\right\}_{i=0}^N$ be an $\varepsilon$-pseudo-orbit, satisfying the assumptions of the lemma. In particular, suppose that for some fixed $0\leq i<N$,  $\mathcal{L}_{\varepsilon}^if(x_i)>0$. Then
$$
\mathcal{L}_{\varepsilon}^{i+1}f(x_{i+1})=\int_M \mathcal{L}_{\varepsilon}^if(y)p_{\varepsilon}(y, x_{i+1})\:dy.
$$
By the hypothesis (P2) we  have $x_{i+1}\in B_{\varepsilon}(Ty) \Rightarrow p_{\varepsilon}(y,x_{i+1})>0$ and by the preceding continuity assumptions  there exists $\delta>0$ such that $y\in B_{\delta}(x_i) \Rightarrow  \mathcal{L}_{\varepsilon}^if(y)>0$. But this $\delta$-neighborhood of $x_i$ can be made smaller in such a way that when $y$ belongs to it, $d(Ty,Tx_i)\le \frac{\varepsilon-d(Tx_i,x_{i+1})}{2}$ which implies that $x_{i+1}$ is $\varepsilon$-close to $Ty$.
Therefore for all the $y$ in this $\delta$-neighborhood (which is of positive Lebesgue measure), the integrand above is strictly positive and this finishes the proof of the Lemma.
\qed
\end{proof}
\begin{proof}[of Theorem \ref{MT1}]
We first show that for every least element $\tilde\Lambda$  there exists a neighborhood $\mathcal{U}_{\varepsilon}$ such that $T(\mathcal{U}_{\varepsilon})\subset \mathcal{U}_{\varepsilon}$.  We denote by $B_{\varepsilon}(A)$ the (open) $\varepsilon$-neighborhood of a set $A$, that is, $B_{\varepsilon}(A)=\{x\in M: d(x,A)<\varepsilon\}$.  We observe that even though a least element $\tilde\Lambda$ is forward invariant, the image of a ball of radius $\varepsilon$ centred at a point in $\tilde\Lambda$ may not be necessarily contained in $\tilde\Lambda$. However, this ball will surely be a subset of the open set $\mathcal{U}_{\varepsilon,1}:=B_{\varepsilon}(T\tilde\Lambda)$. We define inductively a family of nested open sets $\mathcal{U}_{\varepsilon,n}:=B_{\varepsilon}(T\mathcal{U}_{\varepsilon,n-1})$ and consider the open neighborhood $\mathcal{U}_{\varepsilon}$ of $\tilde\Lambda$ defined by $\mathcal{U}_{\varepsilon}:=\cup_{n=1}^{\infty}\mathcal{U}_{\varepsilon,n}$. This set is clearly forward invariant under $T$, and its closure is disjoint, for $\varepsilon$ small enough, from the supports of the other ergodic acim; otherwise, we can construct an $\varepsilon$ pseudo-orbit linking the least element to them.

Let ${\mathcal B}(\mathcal{U}_{\varepsilon}):=\{f\in\mathcal B| f \textnormal{ is supported on }\mathcal{U}_{\varepsilon}\}$. Since $\mathcal{L}_{\varepsilon}f(x)=\int_Mp_{\varepsilon}(y,x)f(y)dy$, the forward invariance of $\mathcal{U}_{\varepsilon}$ together with $x\in B_{\varepsilon}(Ty)$, insure that $\mathcal L_{\varepsilon}$ leaves the Banach space ${\mathcal B}(\mathcal{U}_{\varepsilon})$ invariant.
Then by applying on ${\mathcal B}(\mathcal{U}_{\varepsilon})$ the Lasota-Yorke inequality and successively the Ionescu-Tulcea-Marinescu spectral theorem, (see for instance \cite{BG, HH},) we obtain a fixed point $\mathcal L_{\varepsilon}h_{\varepsilon}=h_{\varepsilon}\in {\mathcal B}(\mathcal{U}_{\varepsilon})$: the measure $\mu_{\varepsilon}=h_{\varepsilon}m$, $\int_{\mathcal{U}_{\varepsilon}}h_{\varepsilon}dm=1$ is therefore stationary.

We now prove that $h_{\varepsilon}$ is the only fixed point of $\mathcal{L}_{\varepsilon}$ in ${\mathcal B}(\mathcal{U}_{\varepsilon})$. Suppose there is another function $h_{\varepsilon}'\in {\mathcal B}(\mathcal{U}_{\varepsilon}))$ with the same property, and let us define the function $\hat{h}=\min(h_{\varepsilon},h_{\varepsilon}')$. Then clearly:
$\min(h_{\varepsilon},h_{\varepsilon}')\leq h_{\varepsilon}$ and $\min(h_{\varepsilon},h_{\varepsilon}')\leq h_{\varepsilon}'$ and thus $\mathcal{L}_{\varepsilon}(\min(h_{\varepsilon},h_{\varepsilon}'))\leq \mathcal{L}_{\varepsilon}h_{\varepsilon}$ and $\mathcal{L}_{\varepsilon}(\min(h_{\varepsilon},h_{\varepsilon}'))\leq \mathcal{L}_{\varepsilon}h_{\varepsilon}'$ which implies $\mathcal{L}_{\varepsilon}(\min(h_{\varepsilon},h_{\varepsilon}'))\leq \min(\mathcal{L}_{\varepsilon}h_{\varepsilon},\mathcal{L}_{\varepsilon}h_{\varepsilon}')$. But
$\min(\mathcal{L}_{\varepsilon}h_{\varepsilon},\mathcal{L}_{\varepsilon}h_{\varepsilon}')=\min(h_{\varepsilon},h_{\varepsilon}')=\hat{h}$, so $\mathcal{L}_{\varepsilon}\hat{h}\leq \hat{h}$ and therefore $\mathcal{L}_{\varepsilon}\hat{h}=\hat{h}$. Let us consider $h_{\varepsilon}-\hat{h}$. It is a nonnegative function and it satisfies $\mathcal{L}_{\varepsilon}(h_{\varepsilon}-\hat{h})=h_{\varepsilon}-\hat{h}$. By Proposition \ref{prop1} and by taking $\varepsilon$ small enough, we insure that the supports of $h_{\varepsilon}$ and $h'_{\varepsilon}$ will intersect the least element in a Borel set $B$ of positive Lebesgue measure. Starting from almost any point in this set we can attain any other point in $\mathcal{U}_{\varepsilon}$ with a (finite) ${\varepsilon}$-pseudo-obit \footnote{The ergodicity insures this possibility for points in the same ergodic component; the equivalence relation allows to pass from one representative to the other in the least element, and finally the recursive construction of $\mathcal{U}_{\varepsilon}$ allows to get the external points $\mathcal{U}_{\varepsilon}/\Lambda$.}. Take any $x_0\in B\setminus (\mathcal{D}_{h_{\varepsilon}-h}\cap\mathcal{D}_T)$ such that $(h_{\varepsilon}-\hat{h})(x_0)>0$.  For any point $x\in \mathcal{U}_{\varepsilon}\setminus (\mathcal{D}_{h_{\varepsilon}-h}\cap\mathcal{D}_T)$ there is an ${\varepsilon}$-pseudo-orbit $\left\{x_i\right\}_{i=0}^{N}\subset M\setminus (\mathcal{D}_{h_{\varepsilon}-h}\cap\mathcal{D}_T)$ which starts from $x_0$ and lands at $x=x_N$.  Hence we can apply Lemma \ref{pos} with $f=h_{\varepsilon}-\hat{h}$ to get $\mathcal{L}_{\varepsilon}^k(h_{\varepsilon}-\hat{h})(x_k)=(h_{\varepsilon}-\hat{h})(x_k)>0$, $k=1,\cdots ,N$. This implies that  $(h_{\varepsilon}-\hat{h})(x)>0$ for all $x\in \mathcal{U}_{\varepsilon}\setminus (\mathcal{D}_{h_{\varepsilon}-h}\cap\mathcal{D}_T)$, that is  $h_{\varepsilon}>h_{\varepsilon}'$ almost everywhere on $\mathcal{U}_{\varepsilon}$, contradicting the fact that $\int h_{\varepsilon}dm=\int h_{\varepsilon}'dm=1$. Therefore $h_{\varepsilon}=h_{\varepsilon}'$ almost everywhere on $\mathcal{U}_{\varepsilon}$.  We get part 1 of the theorem.

Now we prove part 2.  Suppose $\tilde{\Lambda}_0$ is not a least element.  We will show that $h_{\varepsilon}'=0$ on $\tilde{\Lambda}_0$ for the density $h_{\varepsilon}'$ be the density of any acsm.  We proceed by contradiction.  In this case we have $\tilde{\Lambda}_0\rightarrow\tilde{\Lambda}$ for some least element $\tilde{\Lambda}$, if otherwise $\tilde{\Lambda}_0$ would be a least element itself. So there is an $\varepsilon$-pseudo-orbit which starts from $\tilde{\Lambda}_0$ and ends up in $\tilde{\Lambda}$.  Let $h_{\varepsilon}$ denote the density of the unique acsm supported on $\tilde\Lambda$. If $h'_{\varepsilon}\not= 0$, we can invoke again the arguments of Lemma \ref{pos} and part 1 of this theorem to conclude that $h_{\varepsilon}>0$ on $\tilde{\Lambda}$ too and also that $\min(h_{\varepsilon},h_{\varepsilon}')$ is a fixed point of $\mathcal{L}_{\varepsilon}$. But the support of such a minimum is a subset of $\mathcal{U}_{\varepsilon}$ since for $\varepsilon$ small enough, $h_{\varepsilon}=0$ outside $\mathcal{U}_{\varepsilon}$.  Then by uniqueness of the density $h_{\varepsilon}$ over $\mathcal{U}_{\varepsilon}$ we get that $h_{\varepsilon}=\min(h_{\varepsilon},h_{\varepsilon}')$. This implies that $\int_Mh_{\varepsilon}'dm>1$, which is false.
\qed
\end{proof}

\begin{proof}[of Corollary 2]
By Proposition 1 and part~2 of Theorem~\ref{MT1}, for $\varepsilon$ small enough, the support of any ergodic acsm $\mu^{\varepsilon}_1$ intersects the support of a least element. By part~1 of Theorem~\ref{MT1} a small neighborhood of this least element supports an ergodic acsm  $\mu^{\varepsilon}_2$. By repeating the arguments of the proof of part~2 of Theorem~\ref{MT1} we obtain that those two ergodic acsm must coincide. Finally the unicity of the least element inside the support of $\mu^{\varepsilon}_1$  follows from the fact that two disjoints least elements are at  strictly positive distance and therefore they cannot share the same ergodic acsm.
\qed
\end{proof}
\section{Pseudo-orbits and metastability}\label{meta}
An ergodic dynamical system is said to be {\em metastable} if it possesses regions in its phase space that remain close to invariant for long periods of time. A well-known approach for detecting such a behaviour is by proving that the corresponding Perron-Frobenius operator\footnote{In our setting, since we assume that $\mathcal L_{\varepsilon}$ satisfies (RLY), the operator $\mathcal L_{\varepsilon}$ is quasi-compact on $\mathcal B$; i.e., $\exists$ an $r\in(\eta, 1)$ such that outside a ball centred at zero and of radius $r$, the operator $\mathcal L_{\varepsilon}$, as an operator on $\mathcal B$, has only discrete spectrum.} has a subdominant real eigenvalue $\xi_{\varepsilon}$. Then the positive and negative parts of the eigenfunction corresponding to $\xi_{\varepsilon}$ can be used to identify sets which remain close to invariant for long periods of time. Such sets are often called {\em almost invariant sets}. For more information on almost invariant sets we refer the reader to \cite{FS0} and references therein. An analogous theory also exists in the framework of non-autonomous dynamical systems, where the analogous sets are called {\em coherent structures} (see for instance \cite{FS} and references therein).

\bigskip

\noindent In this section we assume that:\\
\noindent (M1) As an operator on $\mathcal B(M)$, $\mathcal L$ has $1$ as an eigenvalue of multiplicity two. Moreover, if $\lambda\not= 1$ is an eigenvalue of $\mathcal L$, then $|\lambda|<1$.\\
\noindent (M2) The map $T$ has a unique least element.\\

Under conditions (M1) and (M2), we will show that random perturbations of $T$ exhibit a \textit{metastable behavior}. In particular, we will show that $\mathcal L_{\varepsilon}$, as an operator on $\mathcal B$, will have $1$ as a \textit{simple} eigenvalue and will have another real eigenvalue $\xi_{\varepsilon}$ close to $1$. Moreover $\xi_{\varepsilon}$ has the second largest modulus among eigenvalues of $\mathcal L_{\varepsilon}$. Such a $\xi_{\varepsilon}$ determines the rate of mixing \cite{Ba} of the random system $(\mathcal{X}^{\varepsilon}_n)$.\\

 For this purpose, we first introduce some notation and recall the Keller-Liverani perturbation theorem \cite{KL1}. We adapt it to our situation which deals with
  the two adapted norms $||\cdot||_{\mathcal B}=|\cdot|+||\cdot||_1$ and $||\cdot||_1$. For the unperturbed Perron-Frobenius operator ${\mathcal L}$ let us
consider the set
$$V_{\delta,r}({\mathcal L})=\{z\in C: |z|\le r\textnormal{ or dist}(z,\sigma({\mathcal L}))\le\delta\},$$
where $\sigma({\mathcal L})$ is the spectrum of ${\mathcal L}$ as an operator on $\mathcal B$.
Further, for $\varepsilon\ge 0$, we define the following operator norm
\begin{equation}\label{3norm}
|||{\mathcal L}_{\varepsilon}|||=\sup_{||f||_{\mathcal B}\le 1}||{\mathcal L}_{\varepsilon}f||_1.
\end{equation}
Conditions (LY) and (RLY) are necessary for the operators $\mathcal L$, $\mathcal L_{\varepsilon}$ to satisfy the assumptions \cite{KL1}. Thus, we are ready to state and use the following important result of \cite{KL1}:
\begin{theorem}\cite{KL1}\label{ThKL99}
If $\lim_{\varepsilon\to 0}|||{\mathcal L}_{\varepsilon}-{\mathcal L}|||=0$  then
for sufficiently small $\varepsilon>0$, $\sigma({\mathcal L}_{\varepsilon})\subset
V_{\delta,r}({\mathcal L})$. Moreover, in each connected component
of $V_{\delta,r}({\mathcal L})$ that does not contain $0$ both
$\sigma({\mathcal L})$ and $\sigma({\mathcal L}_{\varepsilon})$ have the same
multiplicity; i.e., the associated spectral projections have the
same rank.
\end{theorem}
Using Theorem \ref{ThKL99}, we show that our random system $(\mathcal X_{\varepsilon}^n)$ exhibits metastable behavior:
\begin{proposition}\label{mixing}
Suppose that
\begin{itemize}
\item $T$ satisfies assumptions (M1), (M2);
\item $\lim_{\varepsilon\to 0} |||\mathcal L-\mathcal L_{\varepsilon}|||=0$.
\end{itemize}
Then, as an operator on $\mathcal B$, $\mathcal L_{\varepsilon}$ has $1$ as a simple eigenvalue. Moreover, $\mathcal L_{\varepsilon}$ has a real eigenvalue $\xi_{\varepsilon}$ very close to $1$. In particular, $\xi_{\varepsilon}$ has the second largest modulus among eigenvalues of $\mathcal L_{\varepsilon}$.
\end{proposition}
\begin{proof}
Assumption (M1) states that the spectrum of
$\mathcal L$, as an operator on $\mathcal B$, satisfies the following:
$\exists$ an $r\in(\eta, 1)$\footnote{By (LY) and (RLY), $\eta$ is an upper bound on
the essential spectral radius of $\mathcal L$ and the essential spectral radius of
$\mathcal L_{\varepsilon}$.} and a $\delta>0$ such
that:
\begin{enumerate}
\item The eigenvalue 1 of $\mathcal L$ is of multiplicity two;
\item if $\lambda_i\not= 1$ is an eigenvalue of $\mathcal L$, then $\lambda_i\in B(0,r)$;
\item $B(0,r)\cap B(1,\delta)=\emptyset$.
\end{enumerate}
Moreover, under assumptions (M2), using Theorem \ref{MT1}, the random map
$(\mathcal X_{\varepsilon}^n)$ has exactly one ergodic acsm; i.e., as an
operator on $\mathcal B$, $\mathcal L_{\varepsilon}$ has 1 as a \textit{simple}
eigenvalue. Consequently, by Theorem \ref{ThKL99}, for sufficiently
small $\varepsilon$, the spectrum of $\mathcal L_{\varepsilon}$ satisfies the
following:
\begin{enumerate}
\item $\mathcal L_{\varepsilon}$ has a real eigenvalue $\xi_{\varepsilon}<1$, with $\xi_{\varepsilon}\in B(1,\delta)$;
\item if $\lambda_{i,\varepsilon}\notin\{1,\xi_{\varepsilon}\}$ is an eigenvalue of $\mathcal L_{\varepsilon}$, then $\lambda_{i,\varepsilon}\in B(0,r)$.\qed
\end{enumerate}
\end{proof}

\begin{remark}The condition $\lim_{\varepsilon\to 0} |||\mathcal L-\mathcal L_{\varepsilon}|||=0$ of Proposition \ref{mixing} can be checked in several cases. A general theorem is presented in
Lemma 8 of \cite{BaYo} for piecewise expanding maps of the interval endowed with our pair of adapted spaces where the noise is represented by a convolution kernel. In the multidimensional case, using quasi-H\"older spaces, the proof is given in Proposition 4.3 of \cite{AFV}. It should be noted that the previous result implies that the non-essential spectrum of $\mathcal L$ is stable \cite{KL1}.
\end{remark}
\begin{remark}
The technique followed to prove Proposition \ref{mixing} does not work when the number of ergodic $T$-acim is $l> 2$. This is due to the fact that\\
a) If $l$ is odd, then $l-1$ is even. Therefore, the transfer operator $\mathcal L_{\varepsilon}$ may have $l-1$ complex eigenvalues of modulus one sitting in $B(1,\delta)$.\\
b) If $l$ is even, then $l-2$ is even. Therefore, the transfer operator $\mathcal L_{\varepsilon}$ may have $l-2$ complex eigenvalues of modulus one sitting in $B(1,\delta)$.
\end{remark}
Whether Proposition \ref{mixing} is true or not for $l>2$ is an interesting question.
\section{Random Transformations}\label{comp}
\begin{center}
\begin{figure}\label{fig1}
 \hskip -1cm
 \includegraphics[width=4in]{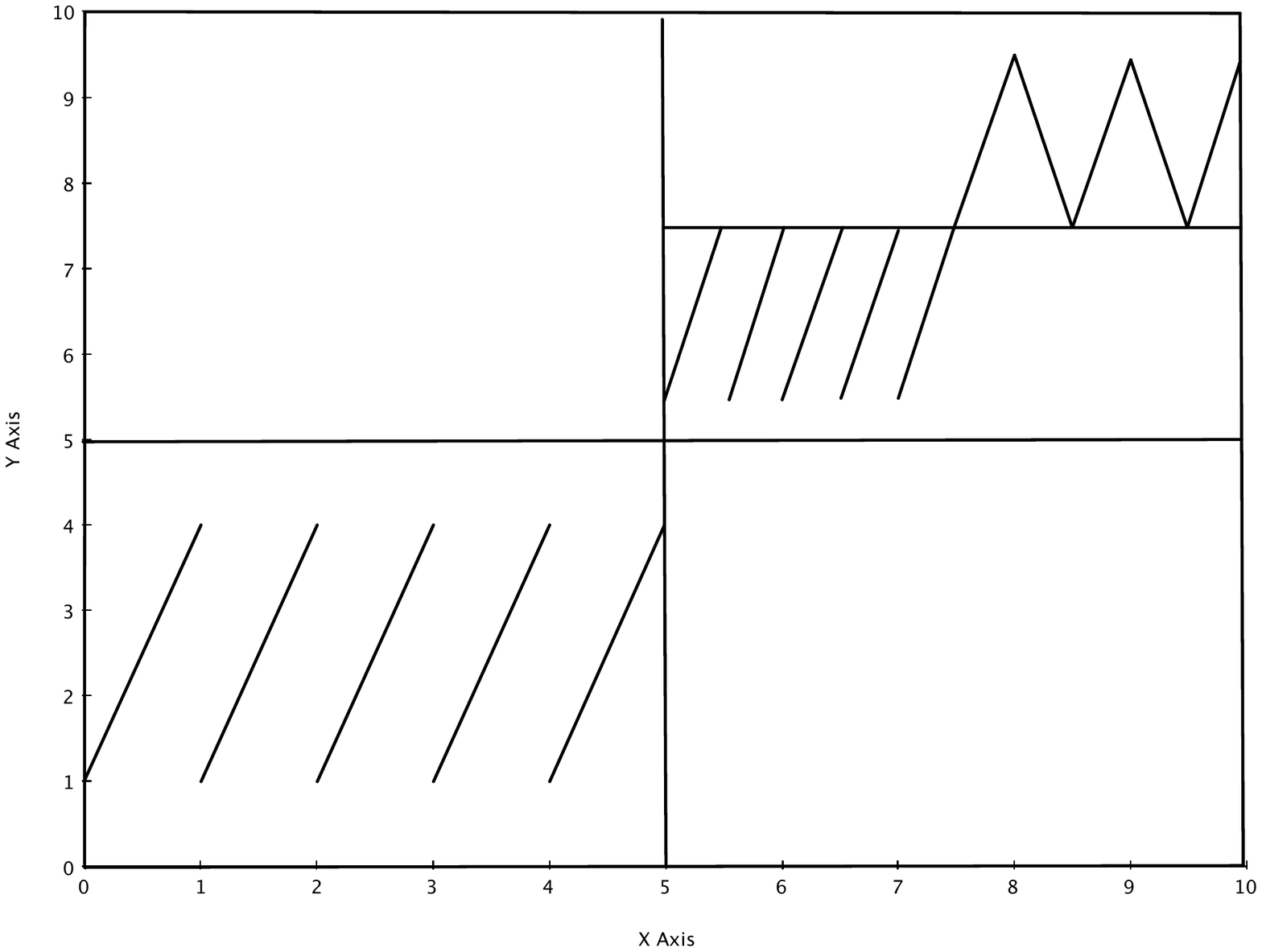}
 \vskip -4cm
   \caption{An example of a $1-$dimensional map $T$ with two least elements}
\end{figure}
\end{center}
In sections \ref{main}, \ref{pf} and \ref{meta} we studied random perturbations of a dynamical system in the framework of general Markov processes. Nevertheless, it is often useful to deal with the case  when the Markov process is generated by \textit{random transformations} \cite{K}. In this setting, we consider an i.i.d.  stochastic process $(\omega_k)_{k\in N}$ with values in $\Omega_{\varepsilon}$ and with probability distribution $\theta_{\varepsilon}$. We associate with each $\omega\in\Omega_{\varepsilon}$ a map $T_{\omega}: M\to M$ and we consider the random orbit starting from the point $x$ and generated by the realization $\underline{\omega}_n=(\omega_1,\omega_2,\cdots,\omega_n)$, defined as : $T_{\underline{\omega}_n}:=T_{\omega_n}\circ\cdots\circ T_{\omega_1}(x)$. This defines a Markov process  $\mathcal X_{\varepsilon}$ with transition function
\begin{equation}\label{gre}
 P (x, A)=\int_{\Omega_{\varepsilon}}\mathbf{1}_{A}(T_{\omega}(x))d\theta_{\varepsilon}(\omega),
 \end{equation}
where $A\in\mathcal B(M)$, $x\in M$ and $\mathbf{1}_{A}$ is the indicator function of a set $A$. The transition function  induces an operator ${\mathcal U}_{\varepsilon}^*$ which acts on measures $\mu$ on $(I,\mathcal B(M))$ as:
$${\mathcal U}_{\varepsilon}^*\mu(A)=\int_{M}\int_{\Omega_{\varepsilon}}\mathbf{1}_{A}(T_{\omega}(x))d\theta_{\varepsilon}(\omega)d\mu(x)=\int_{I}{\mathcal U}_{\varepsilon}{\bf 1}_A(x)d\mu_{\varepsilon}(x),$$ where ${\mathcal U}_{\varepsilon}$ is the random evolution operator acting on $L^{\infty}_m$ functions $g$:
\begin{equation}\label{koop}
{\mathcal U}_{\varepsilon}g=\int_{\Omega_{\varepsilon}}g\circ T_{\omega}d\theta_{\varepsilon}(\omega).
\end{equation}
A measure $\mu_{\varepsilon}$ on $(M,\mathcal B(M))$ is called a $\mathcal X_{\varepsilon}$-stationary measure  if and only if, for any $A\in\mathcal B(M)$,
\begin{equation}\label{stationary}
{\mathcal U}_{\varepsilon}^*\mu_{\varepsilon}(A)=\mu_{\varepsilon}(A).
\end{equation}
We are interested in studying $\mathcal X_{\varepsilon}$-acsm. By (\ref{koop}), one can define the transfer operator $\mathcal L_{\varepsilon}$ (Perron-Frobenius) acting on $L^{1}(M,\mathcal B(M), m)$ by:
\begin{equation}\label{RPF}
(\mathcal L_{\varepsilon} f)(x)=\int_{\Omega_{\varepsilon}}\mathcal L_{\omega} f (x)d\theta_{\varepsilon}(\omega),
\end{equation}
which satisfies the duality condition
\begin{equation}\label{du}
\int_M g \mathcal L_{\varepsilon} f dm = \int_M {\mathcal U}_{\varepsilon}g f dm
\end{equation}
where $g$ is in $L^{\infty}_m$ and $\mathcal L_{\omega}$ is the transfer operator associated with $T_{\omega}$.
It is well know that $\mu_{\varepsilon}:=\rho_{\varepsilon}m$ is a $\mathcal X_{\varepsilon}$-acsm if and only if $\mathcal L_{\varepsilon}\rho_{\varepsilon}=\rho_{\varepsilon}$; i.e., $\rho_{\varepsilon}$ is a $\mathcal X_{\varepsilon}$-invariant density. In order to use the results of sections \ref{main}, \ref{pf} and \ref{meta}, we also assume that assumptions (P1), (P2)\footnote{ (P1) and (P2) in this setting are analogous to:\\ (P1) For all $x\in I$ the measure $P(x,\cdot)$ defined above on the Borel subsets of $I$ by $ P(x,A)=\theta_{\varepsilon}\{\omega\in \Omega_{\varepsilon}\; ,\ T_{\omega}(x)\in A\}$ is absolutely continuous with respect to Lebesgue, namely we have a summable density $p_{\varepsilon}(x,\cdot)$ such that: $P(x,A)=\int_A p_{\varepsilon}(x,y)dy$;\\
(P2) We have:  support of $P(x,\cdot)$ coincides with $B_{\varepsilon}(Tx), \forall x\in I$.} (B1)-(B3), (LY) and (RLY) hold. Moreover, we assume that (\ref{RPF}) reduces to
\begin{equation}\label{red}
(\mathcal L_{\varepsilon} f)(x)=\int_{\Omega_{\varepsilon}}\mathcal L_{\omega} f (x)d\theta_{\varepsilon}(\omega)=\int_M p_{\varepsilon}(z,x)f(z)dz.
\end{equation}
In fact, an important example of a random perturbation where $\mathcal L_{\varepsilon}$ can be reduced as in (\ref{red}) is the case of {\em additive} noise.  For instance if $M=S^q$, the $q$-dimensional torus \footnote{If $M$ is not the torus we assume that, for all $\omega\in \Omega_{\varepsilon}$, $T_{\omega}(M)\subseteq M$.}, let define $T_{\omega}= T(x)-\omega$ mod $S^q$,
where $\omega\in S^q$. Let the density of $\theta_{\varepsilon}$ with respect to the Lebesgue measure $d\omega$
on $S^m$, $h_{\varepsilon}$, be continuously differentiable with support contained in the square
$\Omega_{\varepsilon}\equiv[-\varepsilon,\varepsilon]^q$: $\int d\theta_{\varepsilon}=\int h_{\varepsilon}(\omega)d\omega=1$. It is then straightforward to check that $p_{\varepsilon}(x,y)=
 h_{\varepsilon}(Tx-y)$.
 \begin{remark}\label{RR}
 The preceding example illustrates very well the relation between the two approaches used in this paper to deal with randomness. Namely the {\em Markov chain} approach, which was used to prove Theorem 1, and the {\em random transformations} approach which permits to follow the orbit of a point under the concatenation of the randomly chosen maps. Consequently, the latter allows for more explicit representation of objects like the evolution operator and the transfer operator. In fact, relation (\ref{red}) is a general fact whenever the transition function for the Markov chain is given by the integral (see (\ref{gre})), $P (x, A)=\int_{\Omega_{\varepsilon}}\mathbf{1}_{A}(T_{\omega}(x))d\theta_{\varepsilon}(\omega)$, and the noise is {\em absolutely continuous}, namely $\theta_{\epsilon}(\omega\in \Omega_{\epsilon}; \ T_{\omega}x\in A)=\int_A p_{\epsilon}(x,z)dz$, where $A$ is a measurable set in $M$. This in particular means that we can always construct a Markov chain starting with a random transformation. The converse is also true. We refer to  \cite{K} for the construction, and to \cite{JKR} for recent results in this direction\footnote{\cite{JKR} gives a representation of a local Markov chain perturbation by random diffeomorphisms close to the unperturbed one. It would be interesting to extend the work of \cite{JKR} to $C^2$ endomorphisms where the corresponding random expanding map satisfy (RLY).}. Finally, we stress that Theorem~1 has been proved for an absolutely continuous noise. This means that it cannot be applied, in its actual form, to random perturbations on a finite noise space, for which the probability $\theta_{\epsilon}$ is an atomic measure and the random transfer operator $\mathcal{L}_{\epsilon}$ applied to the function $f$ becomes a weighted operator of the form $\sum_{j\ge 1} \mathcal{L}_{\omega_j}\ f  \ p_{\omega_j}$, where the weights $p_{\omega_j}, \  \sum_{j\ge 1}p_{\omega_j}=1$, are associated to the values of the random variables $\omega_j$.
 \end{remark}
 \subsection{$1$-dimensional examples} To illustrate our results, we  present two simple examples of  $1$-dimensional maps. In these examples the Banach space $\mathcal B$ is considered to be the space of functions of bounded variation. In Example \ref{Ex1} we present a map that has three ergodic components and two least elements. 
 \begin{example}\label{Ex1}
 In this example $T:[0,10]\to[0,10]$. The graph of $T$ is shown in Fig. 1. $T$ is piecewise linear and Markov with respect to the partition: $$[0,1),\dots,[4,5), [5,5.5)\dots,[9.5,10].$$
 One can easily check  that $T$ has exactly three ergodic acim whose supports $\Lambda_1, \Lambda_2, \Lambda_3$ are, respectively, equal to $[1,4]$, $[5.5,7.5]$, and $[7.5,9.5]$. Moreover, one can easily check that $T$ admits two least elements. Namely, $\{\Lambda_1\}$ and $\{\Lambda_2,\Lambda_3\}$.
 \end{example}
 \subsection{ A $2$-dimensional example}
 \begin{center}
\begin{figure}
\hskip -1cm
\includegraphics[width=4in]{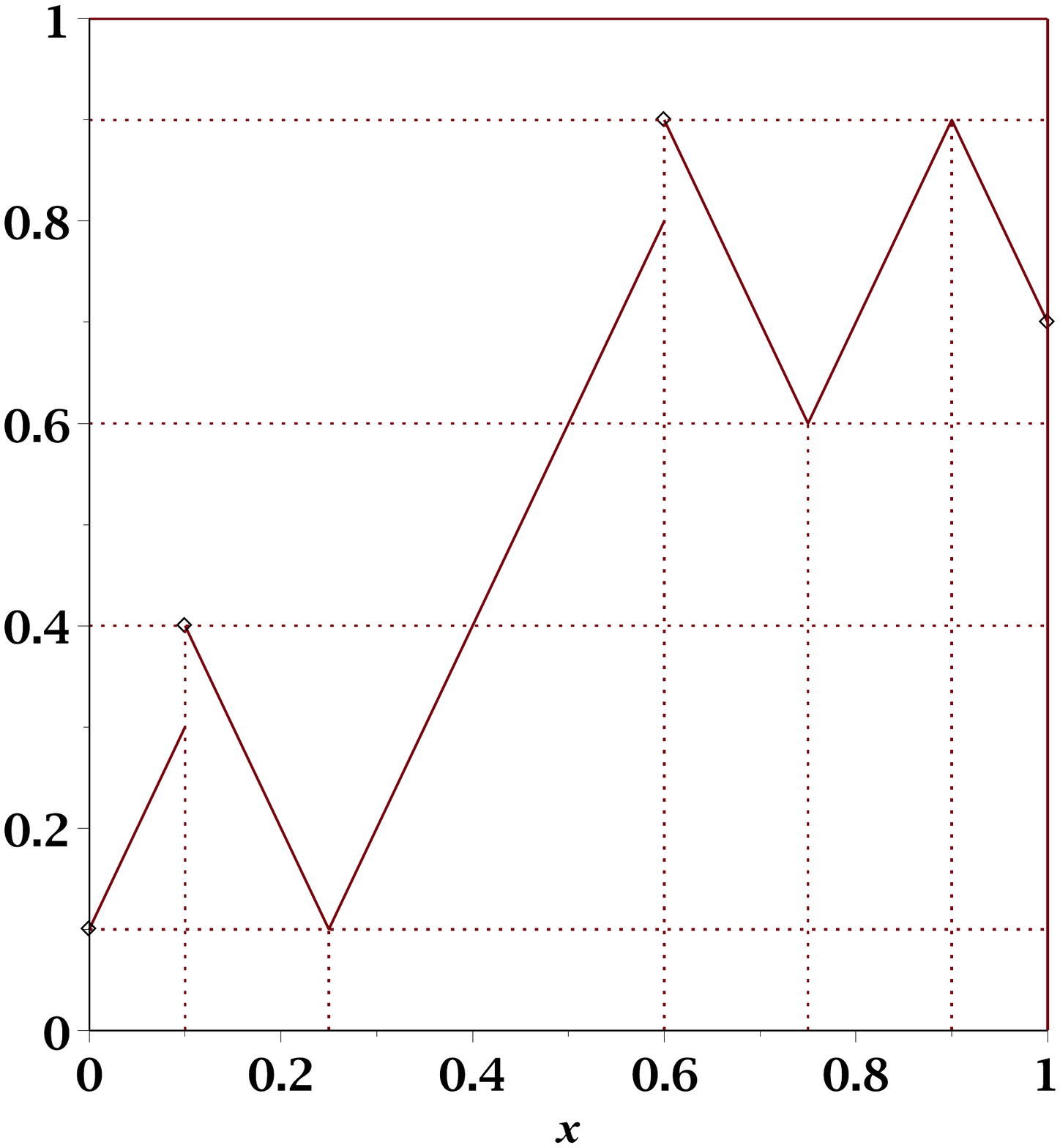}
 \vskip -2cm
   \caption{The graph of the map $T$ for $a=0.1$. The least element for $T$ is the closed interval $[\frac12+a, 1-a]$. This induces the least element $[\frac12+a, 1-a]\times S $ for $\Phi_0$ on $[0,1]\times S$.}
\end{figure}
\end{center}
\begin{center}
\begin{figure}\label{fig4}
 \hskip -1cm
 \includegraphics[width=6in]{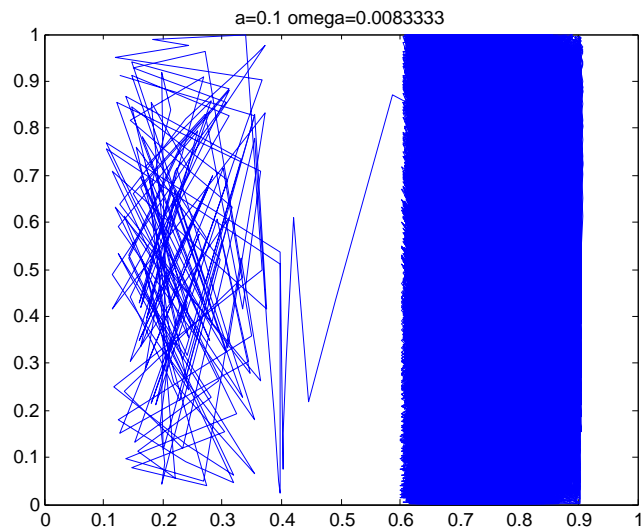}
  \caption{The least element on the right side $x>\frac12$ for the map $\Phi_{\omega}$; notice that $\omega$ on the top of the figure is just $\epsilon=1/120.$}
\end{figure}
\end{center}
 \begin{example}
 We now present an example in higher dimensions, in particular the two-dimensional skew system $\Phi_{\omega}: [0,1]\times S\rightarrow  [0,1]\times S $ defined as $(x',y')=\Phi_{\omega}(x,y)$ where
$$
x'= T(x) + \omega y
$$
$$
y'=2y \ \text{ mod}1,
$$
where $S$ denotes the unit circle, and $T:[0,1]\to[0,1]$ is given by

\begin{itemize}
\item for $x\in[0,a), \ T(x)= (\frac{1}{2a}-3)x+a$;
\item for $x\in [a, \frac14), \ T(x)= -2x +(a+\frac12)$;
\item for $x\in [\frac14, \frac12+a), \ T(x)= 2x+(a-\frac12)$;
\item for $x\in [\frac12+a, \frac34), T(x)= -2x+(2+a)$;
\item for $x\in [\frac34, 1-a), T(x)= 2x+ (a-1)$;
\item for $x\in [1-a, 1], T(x)= (3-\frac{1}{2a})x+(-\frac52+\frac{1}{2a}+2a)$.
\end{itemize}
The graph of $T$ is depicted in Fig.2  \\

We considered a piecewise linear map to simply the exposition. However, this is not really needed to apply our results. In fact, a map with nonlinear branches can be used in the example as long as we keep uniform dilatation, bounded distortion, and a $C^{1+}$ smoothness.   For each $\omega$ we have a different random map and we compose them by taking $\omega$ uniformly distributed, for instance, between  $(-\epsilon, \epsilon)$. In this case $\theta_\epsilon(\omega)=\frac{1}{2\epsilon}d\omega$.  The positive parameter $a$ can be chosen equal to $\frac{1}{10}$, in such a way that the image of the unit interval remains in $[0,1]$ when $\epsilon <a$. It is very easy to check that the unperturbed map $\Phi_{0}$ has two ergodic components which are, respectively, subsets of $[a, \frac12-a]\times S$ and $[\frac12+a, 1-a]\times S$ and the latter is a least element. These ergodic components are with respect to absolutely continuous invariant measures whose densities are the fixed points of the Perron-Frobenius operator associated to $\Phi_0$.  The existence of such fixed points follow easily by obtaining a Lasota-Yorke inequality on a suitable function space, such as the space of functions of bounded variation \cite{BG} or  quasi-H\"older functions, \cite{S,HV}, which satisfy all the assumptions required in this paper. Notice that a Lasota-Yorke inequality can be obtained as well for the random Perron-Frobenius operator associated to the random system, by using the closeness of the perturbed maps $\Phi_{\omega}, \ |\omega|\le \epsilon$ for small $\epsilon$ (this means that the constants $\eta$ and $D$ in (LY) and (RLY) can be chosen to be the same for the unperturbed and the perturbed systems\footnote{For these kind of uniformly expanding maps those factors are basically related to the multiplicity of the intersection of the discontinuous lines, which is $2$ in this example, and to the norm and the determinant of the Jacobian matrix of $\Phi_{\omega}$, which is $D\Phi_{\omega}=\begin{pmatrix}
T'(x)&\omega\\
0&2 \\
\end{pmatrix}$, where $|T'|\in\{|3-\frac{1}{2a}|, 2\}$. The determinant does not depend on the noise and (any) norm of the matrix can be chosen uniformly bounded for $\epsilon$ small enough. }). According to our main theorem, there must be only one ergodic absolutely continuous stationary measure and this must be supported in a neighborhood of the least element. In Fig. 4 we show the limit set of several random orbits taken with
$\epsilon=1/120$ and $a=0.1$.  All these random orbits accumulate in the right hand side of $[0,1]\times S$ as predicted by our theory.

\end{example}
\bibliographystyle{amsplain}

\end{document}